\documentclass[9pt,a4paper, leqno]{amsart}

\usepackage{amssymb,amsmath,amsfonts,amsthm,a4}
\usepackage{graphicx, amsbsy, mathrsfs, esint, bbm, dsfont, bm}
\usepackage{color}
\usepackage{mathtools}
\usepackage{eucal}
\usepackage{slashed}

\usepackage{cite}
\usepackage{enumerate}

\usepackage[titletoc,title]{appendix}

\usepackage{graphicx}   
\usepackage{subfigure}  
\usepackage{multirow}
\usepackage{slashed} 

\usepackage{hyperref}
\newtheorem{theorem}{Theorem}
\newtheorem{lemma}{Lemma}[section]

\newtheorem{proposition}{Proposition}[section]
\newtheorem{assumption}{Assumption}
\newtheorem*{conjecture*}{Conjecture}

\newtheorem*{assumption*}{Assumption}
\theoremstyle{remark}

\newtheorem*{remarks*}{Remarks}
\newtheorem*{remark*}{Remark}

\newcommand{\R}{\mathbb{R}} 
\newcommand{\C}{\mathbb{C}} 
\newcommand{\N}{\mathbb{N}} 
\newcommand{\eps}{\varepsilon} 
\newcommand{\dHsN}[2]{\| {#2} \|_{\dot{H}^{#1}}} 
\newcommand{\LpN}[2]{\left\| {#2} \right\|_{L^{#1}}} 

\newcommand{\nequiv}{\not \equiv}
\newcommand{\widebar}{\overline}
\newcommand{\diff}{\, d} 
\newcommand{\comment}[1]{}
\numberwithin{equation}{section}

\newcommand{\JJ}{\mathcal{J}_{\vB,\omega,\sigma}}

\newcommand{\JJJ}{\JJ^*}
\newcommand{\norm}[1]{\left \|#1 \right \|}
\newcommand{\LpNmu}[2]{\| {#2} \|_{L^{#1}_\mu}} 

\newcommand{\be}{\begin{equation}}
\newcommand{\ee}{\end{equation}}
\newcommand{\bes}{\begin{equation*}}
\newcommand{\ees}{\end{equation*}}

\newcommand{\eB}{{\mathbf{e}}}
\newcommand{\vB}{{\mathbf{v}}}

\newcommand{\pt}{\partial}
\newcommand{\Rs}{\mathsf{R}}

\newcommand{\ii}{\mathrm{i}}
\newcommand{\weakto}{\rightharpoonup}
\newcommand{\FF}{\mathcal{F}}

\catcode`@=11
\def\section{\@startsection{section}{1}%
  \z@{1.5\linespacing\@plus\linespacing}{.5\linespacing}%
  {\normalfont\bfseries\large\centering}}
\catcode`@=12

\begin{document}

\title[On Symmetry of Traveling Solitary Waves]{On Symmetry of Traveling Solitary Waves \\ for dispersion generalized NLS}

\begin{abstract}
We consider dispersion generalized nonlinear Schr\"odinger equations (NLS) of the form
$$
i \pt_t u = P(D) u - |u|^{2 \sigma} u,
$$
where $P(D)$ denotes a (pseudo)-differential operator of arbitrary order. As a main result, we prove symmetry results for traveling solitary waves in the case of powers $\sigma \in \N$. The arguments are based on Steiner type rearrangements in Fourier space. Our results apply to a broad class of NLS-type equations such as fourth-order (biharmonic) NLS, fractional NLS, square-root Klein-Gordon and half-wave equations.
\end{abstract}

\author[L. Bugiera]{Lars Bugiera}
\address{University of Basel, Department of Mathematics and Computer Science, Spiegelgasse 1, CH-4051 Basel, Switzerland.}
\email{lars.bugiera@unibas.ch}

\author[E. Lenzmann]{Enno Lenzmann}
\address{University of Basel, Department of Mathematics and Computer Science, Spiegelgasse 1, CH-4051 Basel, Switzerland.}
\email{enno.lenzmann@unibas.ch}

\author[A. Schikorra]{Armin Schikorra}
\address{University of Pittsburgh, Department of Mathematics, 301 Thackeray Hall, Pittsburgh, PA 15260, USA}
\email{armin@pitt.edu}

\author[J. Sok]{J\'er\'emy Sok}
\address{University of Basel, Department of Mathematics and Computer Science, Spiegelgasse 1, CH-4051 Basel, Switzerland.}
\email{jeremyvithya.sok@unibas.ch}

\maketitle


\section{Introduction and Main Results}

The aim of the present paper is to derive symmetry results for traveling solitary waves for nonlinear dispersive equations of nonlinear Schr\"odinger (NLS) type. As a model case in space dimension $n \geq 1$, we consider equations of the form
  \be \tag{gNLS} \label{eq:gNLS}
   \ii \partial_t u = P(D) u - |u|^{2\sigma}u
  \ee
 for functions $u : [0,T) \times \R^n \to \C$. Here $P(D)$ denotes a self-adjoint and constant coefficient (pseudo-)differential operator defined by multiplication in Fourier space as
 \be
 \widehat{(P(D) u)}(\xi) = p(\xi) \widehat{u}(\xi),
 \ee
where suitable assumptions on the multiplier $p(\xi)$ will be stated below. In fact, the class of allowed symbols $p(\xi)$ will be rather broad including e.\,g.~fractional and polyharmonic NLS, higher-order NLS with mixed dispersions, half-wave and square-root Klein-Gordon equations (see, e.\,g.~\cite{BoCa-18,BoHiLe-16,Fi-15,FiIlPa-02,GeLePoRa-18,KaSh-00,KrLeRa-13}) and also Subsection 5.1 below.

Let us first make with some general remarks. Due to the focusing nature of the nonlinearity in \eqref{eq:gNLS}, we expect the existence of solitary waves $u(t,x) = e^{\ii t \omega}Q(x)$. In fact, by the translational invariance exhibited by the problem at hand, we expect that \textbf{traveling solitary waves} exist, which by definition are solutions of the form
\be \label{eq:u_sol}
u(t,x) = e^{\ii \omega t} Q_{\omega, \vB} (x- \vB t)
\ee
with some non-trivial profile $Q : \R^n \to \C$ depending on the given parameters $\omega \in \R$ (frequency) and $\vB \in \R^n$ (velocity). However, except for the important but special case of classical NLS when $P(D) = -\Delta$ and its Galilean invariance (see \eqref{eq:Galilei} below), there is no known {\em boost symmetry}, which transforms a solitary wave at rest with $\vB = 0$ into a traveling solitary wave with $\vB \neq 0$ for a general NLS-type equation like \eqref{eq:gNLS}. More importantly, in the absence of an explicit boost transform, the symmetries of the profile function $Q_{\omega, \vB}$ remain elusive in general. Yet, by inspecting the known explicit case when $P(D)=-\Delta$, we may conjecture that the following symmetries are also present in the general case: Up to translation and complex phase, i.\,e., replacing $Q_{\omega, \vB}$ by $e^{\ii \theta} Q_{\omega, \vB}(\cdot + x_0)$ with constants $\theta \in \R$ and $x_0 \in \R^n$, we have that:
\begin{enumerate}
\item[(S1)] $Q_{\omega, \vB}$ is \textbf{cylindrically symmetric with respect to} $\vB \in \R^n$, $n \geq 2$, i.\,e., we have
$$
Q_{\omega, \vB}(x) = Q_{\omega, \vB}(\Rs x) \quad \mbox{for all $\Rs \in \mathrm{O}(n)$ with $\Rs \vB = \vB$}.
$$ 
\item[(S2)] We have the \textbf{conjugation symmetry} given by
$$
Q_{\omega, \vB}(x) = \overline{Q_{\omega, \vB}(-x)}.
$$
Thus $\mathrm{Re} \, Q_{\omega, \vB}: \R^n \to \R$ and $\mathrm{Im} \, Q_{\omega, \vB} : \R^n \to \R$ are even and odd functions, respectively.
\end{enumerate}
As our main results below, we will establish the symmetry properties (S1) and (S2)  for so-called \textbf{boosted ground states} $Q_{\omega, \vB}$ which are by definition obtained as optimizers for a certain variational problem. In fact, we will show that (under suitable assumptions) that all such boosted ground state {\em must} satisfy (S1) and (S2). Our arguments will be based on rearrangement techniques (Steiner symmetrizations) performed in Fourier space. The core of our argument to obtain such a sharp symmetry result will be based on a topological property of the set $\{ \xi \in \R^n : |\widehat{Q}_{\omega, \vB}(\xi)| > 0 \}$ combined with a recent rigidity result \cite{LeSo-18} obtained for the {\em Hardy--Littlewood majorant problem} in $\R^n$. A more detailed sketch of the proof will be given below.

\subsection{Setup of the Problem}
Let us formulate the assumptions needed for our result. We impose the following conditions on the operator $P(D)$ in \eqref{eq:gNLS}. 
\begin{assumption} \label{ass1}
The operator $P(D)$ has a real-valued and continuous symbol $p : \R^n \to \R$ that satisfies the following bounds
$$
A |\xi|^{2s} + c \leq p(\xi) \leq B |\xi|^{2s} \quad \mbox{for all $\xi \in \R^n$},
$$
with some constants $s \geq \frac 1 2, A> 0 ,B > 0$, and $c \in \R$.
\end{assumption}

Let us assume that $P(D)$ satisfies the assumption above.  We readily deduce the norm equivalence 
$$
\| u \|_{H^s}^2 =  \| (1-\Delta)^{s/2} u \|_{L^2}^2 \simeq \langle u, (P(D) + \lambda) u \rangle = \int_{\R^n} (p(\xi) + \lambda) |\widehat{u}(\xi)|^2 \, d \xi,
$$
where $\lambda > 0$ is a sufficiently large constant. Moreover, we notice that the problem \eqref{eq:gNLS} exhibits (formally at least) conservation of energy  and $L^2$-mass, which are given by
$$
E[u] = \frac{1}{2} \langle u, P(D) u \rangle - \frac{1}{2 \sigma+2} \| u \|_{L^{2 \sigma+2}}^{2 \sigma+2}, \quad M[u] = \| u \|_{L^2}^2.
$$ 
Furthermore, with the real number $s \geq \frac 1 2$ as in Assumption \ref{ass1}, we define the following exponent (not necessarily an integer number) given by
$$
\sigma_*(s,n) := \begin{dcases*} \frac{2s}{n-2s} & \quad if $s < n/2$, \\ +\infty & \quad if $s \geq n/2$, \end{dcases*} 
$$ 
which marks the threshold of energy-criticality for exponents, i.\,e., the range $1 \leq \sigma <  \sigma_*$ corresponds to the energy-subcritical case for problem \eqref{eq:gNLS}. In fact, we will focus on the range in the rest of this paper with some marginal comments on the energy-critical case $\sigma=\sigma_*$ (which of course can occur only if $s < n/2$).

We are interested in traveling solitary waves with finite energy for the model problem \eqref{eq:gNLS}. By plugging the ansatz \eqref{eq:u_sol} into \eqref{eq:gNLS}, we readily find that the profile $Q_{\vB, \omega} \in H^s(\R^n)$ has to be a weak solution of the nonlinear equation 
 \be \label{eq:Q}
P(D) Q_{ \omega, \vB} +  \ii v \cdot \nabla Q_{\omega, \vB} + \omega Q_{\omega, \vB} - |Q_{\omega, \vB}|^{2 \sigma}Q_{\omega, \vB} = 0. 
  \ee  
As briefly mentioned above, there exists a well-known `gauge transform' (corresponding to Galilean boosts in physical terms) for the classical Schr\"odinger, where we can reduce the general case $\vB \in \R^n$ to vanishing velocity $\vB=0$. More precisely, if we consider \eqref{eq:gNLS} with $P(D)=-\Delta$, the Galilean boost transform given by
\be \label{eq:Galilei}
Q(x) \mapsto e^{\frac{i}{2} \vB \cdot x} Q(x)
\ee 
reduces the analysis of \eqref{eq:Q} to the study of the nonlinear equation 
\be \label{eq:Qground}
{-\Delta} Q + \omega_\vB Q - |Q|^{2 \sigma} Q =0  \quad \mbox{with} \quad \omega_{\vB} = \omega + \frac {|\vB|^2}{4},
\ee
where the boost term $\ii v \cdot \nabla$ has been gauged away. An important feature of the Galilean transform \eqref{eq:Galilei} is that preserves the $L^2$-norm $\| Q_\vB \|_{L^2} = \|Q\|_{L^2}$; in fact, it is a unitary transform on $L^2(\R^n)$. 

However, for general dispersion operators $P(D) \neq -\Delta$, no such explicit boost transform in the spirit \eqref{eq:Galilei} is known to exist. Therefore, an alternative approach is needed to deal with more general $P(D)$ in both respects concerning existence and symmetries of non-trivial profiles $Q_{\vB}$.

\subsection{Existence of Traveling Solitary Waves}
We first recall an existence result from \cite{Hi-17} for non-trivial solutions $Q_{v, \omega} \in H^s(\R^n)$ of \eqref{eq:ground-state-equation}. To construct these solutions, we introduce a suitable variational setting as follows. For given $\vB \in \R^n$ and  $\omega \in \R$ (satisfying some conditions below), we define the Weinstein-type functional of the form
\be \label{eq:Weinstein-functional}
    \JJ (u) := \frac{ \left  \langle u, ( P_\vB(D) + \omega) u \right \rangle^{\sigma+1}}{ \| u \|_{L^{2 \sigma+2}}^{2 \sigma +2}}
     \ee
  where $u \in H^s(\R^n)$ with $u \not \equiv 0$. Here and in what follows, we set
\be
    P_\vB(D) := P(D) + \ii \vB \cdot \nabla ,
\ee
which has the multiplier $p_\vB(\xi) = p(\xi) - \vB \cdot \xi$. Recalling that $P(D)$ satisfies Assumption 1 with some $s \geq \frac 1 2$ and $A>0$, it is straightforward to check that
\be \label{def:Sigma_v}
\Sigma_\vB := \inf_{\xi \in \R^n} p_\vB(\xi) = \inf_{\xi \in \R^n} \left \{  p(\xi) -\vB \cdot \xi  \right \} > -\infty,
\ee
provided that either $s > \frac 1 2$ and $\vB \in \R^n$ arbitrary or $|\vB| \leq A$ in the special case $s=\frac 1 2$. We have the following existence result.

\begin{theorem}[Existence of Boosted Ground States \cite{Hi-17}] \label{thm:existence}
Let $n\geq 1$, $\vB \in \R^n$, and suppose that $P(D)$ satisfies Assumption 1 with some constants $s \geq \frac 1 2$ and $A > 0$, where if $s=1/2$, we also assume that $|\vB| < A$ holds.  

Then, for $0 < \sigma < \sigma_*$ and $\omega > -\Sigma_\vB$, every minimizing sequence for $\JJ$ is relatively compact in $H^s(\R^n)$ up to translations in $\R^n$. In particular, there exists some minimizer $Q_{\vB,\omega} \in H^s(\R^n)\setminus\{0\}$, i.\,e.,
  \bes
    \JJ (Q_{\vB,\omega}) = \inf_{u\in H^s(\R^n) \setminus \{0 \}} \JJ(u),
  \ees
 and $Q_{\vB, \omega}$ solves the profile equation \eqref{eq:Q}.
\end{theorem}

\begin{remarks*}
1) Note that for the borderline case when $s=\frac 1 2$ and $|v| = A$ we still have that the $\inf_{0\nequiv f\in H^s(\R^n)} \JJ(f) > -\infty$, but we do not expect this infimum to be attained. For such non-existence result for the (important) special case of the half-wave equations when $P(D)=\sqrt{-\Delta}$ and $|v| \geq 1$, we refer to \cite{BeGeLeVi-19}.

2) Clearly, the variational ansatz using the functional $\JJ$ will break down if $P(D)$ satisfies the bounds in Assumption 1 with some $0  < s < 1/2$. In this case, the boost term $\ii v \cdot \nabla$ cannot be treated as a perturbation of $P(D)$. In this case, we conjecture that the profile equation \eqref{eq:Q} has only trivial solutions in $H^{1/2}(\R^n)$.  

3) The infimum $\Sigma_\vB$ defined in \eqref{def:Sigma_v} corresponds to the bottom of the essential spectrum of the self-adjoint operator $P_v(D)$ acting on $L^2(\R^n)$ with domain $H^{2s}(\R^n)$. For the specific choices $P(D)=(-\Delta)^s$ and $P(D)=(-\Delta + 1)^{s}$, the number $\Sigma_\vB$ can be explicitly calculated using the Legendre transform of the convex maps $\xi \mapsto |\xi|^{2s}$ and $\xi \mapsto (|\xi|^{2s}+1)^s$, respectively. For  details on this, we refer to \cite{Hi-17}.

4) See also \cite{KrLeRa-13, HoSi-15, NaRa-18}, where the existence of boosted ground states for NLS type equations were shown by concentration-compactness methods for fractional NLS when $P(D)=(-\Delta)^s$ in the range $s \in [\frac 1 2, 1)$. 

\end{remarks*}

From now on, we will refer to minimizers of the functional $\JJ$ as {\bf boosted ground states}. Correspondingly, the solutions $u(t,x) = e^{\ii t \omega} Q_{v, \omega}(x-\vB t)$ will be called \textbf{ground state traveling solitary waves.} It is easy to check that any such boosted ground state $Q_{s, v} \in H^{s}(\R^n)$ satisfies the profile equation \eqref{eq:Q} after a suitable rescaling $Q_{s,v} \mapsto \alpha Q_{s,v}$ with some constant $\alpha >0$.

\subsection{Cylindrical and Conjugation Symmetry for $n \geq 2$}
We now turn to our first main symmetry result, which establishes necessary symmetry properties of minimizers for the Weinstein-type functional $\JJ$ in space dimensions $n \geq 2$, under suitable assumptions on $P(D)$ and for integer $\sigma \in \N$. 

In order to prove a symmetry results for minimizers of $\JJ$, we will further develop the Fourier symmetrization method recently introduced in \cite{LeSo-18}. The main idea there is to use symmetric-decreasing rearrangement in Fourier space. In fact, this approach proves to be a useful substitute for standard rearrangement techniques in $x$-space, which are easily seen to fail for a large class of (e.\,g.~higher-order) operators (such as $P(D)=\Delta^2$) or operators with non-radially symmetric Fourier symbols such as $P_\vB(D)$ above. 

From \cite{LeSo-18} we recall the notion of Fourier rearrangement which is defined as
\be
u^\sharp := \mathcal{F}^{-1} \left \{ (\mathcal{F} u)^* \right \} \quad \mbox{for $u \in L^2(\R^n)$ with $n \geq 1$},
\ee
where $f^*$ denotes the symmetric-decreasing rearrangement of a measurable function $f: \R^n \to \C$ vanishing at infinity. For a non-zero velocities, the presence of the boost term $\ii \vB \cdot \nabla$ breaks radially symmetry in general. In this case, all rearrangement operations  that yield spherically symmetric functions (such as $\sharp$ defined above) cannot be applied to the minimization problem for $\JJ(f)$. However, under a suitable assumption on $P(D)$, we still expect to be able to show cylindrical symmetry of minimizers with respect to the direction given by the vector $\vB \neq 0$. Thus we introduce the following notion: We say that  $f : \R^n \to \C$ is {\bf cylindrically symmetric} with respect to a direction $\eB \in \mathbb{S}^{n-1}$ if we have
\be
f(R y) = f(y) \quad \mbox{for a.\,e.~$y \in \R^n$ and all $\Rs \in \mathrm{O}(n)$ with $\Rs \eB = \eB$}.
\ee
For such functions $f$, we will employ some abuse of notation by writing 
$$
f = f(y_{\parallel}, |y_\perp|),
$$
 where we decompose $y \in \R^n$ as $y=y_\parallel + y_\perp$ with $y_\perp$ perpendicular to $\eB \in \mathbb{S}^{n-1}$. For dimensions $n \geq 2$, we now introduce the following rearrangement operation defined as
\be
u^{\sharp_\eB} := \mathcal{F}^{-1} \left \{ (\mathcal{F} u)^{*_\eB} \right \} \quad \mbox{for $u \in L^2(\R^n)$ with $n \geq 2$},
\ee
where $f^{*_\eB} : \R^n \to \R_+$ denotes the \textbf{Steiner symmetrization in $n-1$ codimensions} with respect to a direction $\eB \in \mathbb{S}^{n-1}$, which is obtained by symmetric-decreasing rearrangements in  $n-1$-dimensional planes perpendicular to $\eB$; see Section \ref{sec:rearrange} below for a precise definition. It is elementary to check that $f^{\sharp_\eB}$ is cylindrically symmetric with respect to $\eB$. 

We now formulate the following assumption for $P(D)$.

\begin{assumption} \label{ass2}
The operator $P(D)$ has a multiplier function $p : \R^n \to \R$ which is cylindrically symmetric with respect to some direction $\eB \in \mathbb{S}^{n-1}$. Moreover, the map 
$$
|\xi_\perp| \mapsto p(\xi_\parallel, |\xi_\perp|)
$$ 
is strictly increasing.
\end{assumption}

We have the following general symmetry result. 

\begin{theorem}[Symmetry of Boosted Ground States for $n \geq 2$] \label{thm:sym}
Let $n \geq 2$ and suppose $P(D)$ satisfies Assumptions \ref{ass1} and \ref{ass2} with some $s \geq \frac 1 2$ and $\eB \in \mathbb{S}^{n-1}$. Furthermore, let $\vB  = |\vB| \eB \in \R^n$ and $\omega \in \R$ satisfy the hypotheses in Theorem \ref{thm:existence} and assume  $\sigma \in \N$ is an integer with $0 < \sigma < \sigma_*(n,s)$.

Then any boosted ground state $Q_{\omega, \vB} \in H^s(\R^n)$ is of the form
$$
Q_{\omega, \vB}(x) = e^{\ii \alpha} Q^{\sharp_\eB}(x+x_0)
$$
with some constants $\alpha \in \R$ and $x_0 \in \R^n$. As a consequence, any such $Q_{\omega, \vB}$ satisfies (up to a translation and phase) the symmetry properties {\em (P1)} and {\em (P2)} for almost every $x \in \R^n$.
\end{theorem}

\begin{remark*}
Since the Fourier transform $\widehat{(Q_{\omega, \vB}^{\sharp_\eB})} = |\widehat{Q}_{\omega, \vB}|^{*_{\eB}} \geq 0$ is nonnegative,  we conclude that any boosted ground state $Q_{\omega, \vB}$ is a {\em positive-definite function} in the sense of Bochner, provided we also assume that $\widehat{Q}_{\omega, \vB} \in L^1(\R^n)$ (or more generally a finite Borel measure on $\R^n$). In many examples of interest, it is easy to check that indeed $\widehat{Q}_{\omega, \vB} \in L^1(\R^n)$ holds. Recall that a continuous function $f : \R^n \to \C$ is said to be positive-definite in the sense of Bochner if for any collections of points $x_1, \ldots, x_m \in \R^n$ we have
$$
\sum_{k,l = 1}^m f(x_k - x_l) \overline{z}_k z_l \geq 0 \quad \mbox{for all $\mathbf{z}=(z_1, \ldots, z_m) \in \C^m$},
$$
i.\,e., the complex matrix $[f(x_k-x_k)]_{1 \leq k,l \leq m}$ is positive semi-definite. As a direct consequence, we find that
$$
f(0) \geq |f(x)| \quad \mbox{for all $x \in \R^n$}.
$$
We refer to \cite{ReSi-78} for a discussion of positive-definite functions.
\end{remark*}

First, we briefly sketch the main line of argumentation for proving Theorem \ref{thm:sym}. Using the fact that $\sigma \in \N$ is an integer and by applying the {\em Brascamp--Lieb--Luttinger inequality} (a.k.a.~multilinear Riesz-Sobolev inequality) in Fourier space, we deduce that any boosted ground state $Q_{\omega, \vB} \in H^s(\R^n)$ satisfies
\be \label{ineq:good_J}
\JJ(Q_{\omega, \vB}^{\sharp_\eB}) \leq \JJ(Q_{\omega, \vB}) .
\ee
In particular, we see that $Q_{\omega, \vB}^{\sharp_{\eB}}$ is also a boosted ground state. More importantly, we find that equality  in \eqref{ineq:good_J} holds if and only if 
\be
|\widehat{Q}_{\omega, \vB}(\xi)| = |\widehat{Q}_{\omega, \vB}(\xi)|^{*_\eB}  \quad \mbox{for all $\xi \in \R^n$}.
\ee
This fixes the modulus of the Fourier transform $\widehat{Q}_{\omega, \vB}$, whereas its phase appears is yet completely undetermined. However, the conclusion of Theorem \ref{thm:sym} will follow once we show
\be
\widehat{Q}_{\omega, \vB}(\xi) = e^{\ii (\alpha + \beta \cdot \xi)} |\widehat{Q}_{\omega, \vB}(\xi)|^{*_\eB}
\ee
with some constants $\alpha \in \R$ and $\beta \in \R^n$. In fact, such a ``rigidity result'' about the phase function (i.\,e.~being just an affine function on $\R^n$) can be deduce from the recent result in \cite{LeSo-18} on the Hardy-Littlewood majorant problem in $\R^n$, provided we know that the open set
\be
\Omega = \{ \xi \in \R^n : |\widehat{Q}_{\omega, \vB}(\xi)| > 0 \} 
\ee 
is connected. Establishing this topological fact is the crux of this paper. We remark that in \cite{LeSo-18} where the symmetric-decreasing (Schwarz) symmetrization in $\R^n$ was used, we always have that $\Omega$ is either an open ball or all of $\R^n$; in particular, the set $\Omega$ is connected. However, for the Steiner symmetrization in $n-1$ codimensions needed to define $\sharp_\eB$ it is far from clear that the $\Omega$ is a connected set. Indeed, it is not hard to construct explicit examples of functions $f$ on $\R^n$ such that $|f| = |f|^{*_\eB}$ such that $\{ |f| > 0 \}$ is not connected.

To eventually  show that $\Omega$ above is in fact connected in our case, we will exploit the equation \eqref{eq:Qground} in Fourier space. As a consequence, we find that $\Omega$ must be equal to its \textbf{$m$-fold Minkowski sum} with the integer $m= 2 \sigma + 1$, i.\,e., we have
\be  \label{eq:Mink}
\Omega = \bigoplus_{k=1}^m \Omega : = \{ y_1 + \ldots + y_m : y_k \in \Omega, \; 1 \leq k \leq m \}.
\ee  
The key step is now to establish the connectedness of $\Omega \subset \R^n$ from this information. Surprisingly, we did not succeed in finding a general argument to conclude that any open (non-empty) set $\Omega \subset \R^n$ that satisfies \eqref{eq:Mink} is necessarily connected. However, by additionally using the cylindrical symmetry of $\Omega$, we are able to conclude that the sets $\Omega$ in question are indeed connected. See also the specific argument for the proof of Theorem \ref{thm:conj} below addressing the one-dimensional case $\Omega \subset \R$.

\subsection{Conjugation Symmetry for $n=1$}
In one space dimension, the concept of the symmetrization operation ${\sharp_\eB}$ becomes void. Still, we expect the conjugation symmetry (P2) to hold for boosted ground states in the one-dimensional case. To this end, we define the following operation
$$
f^\bullet = \mathcal{F}^{-1} \left \{ |\mathcal{F} f| \right \} \quad \mbox{for $f \in L^2(\R^n)$}.
$$
We may still ask whether the boosted ground states $Q_{\omega, \vB} \in H^s(\R)$ as given by Theorem \ref{thm:existence} always obey that 
$$ 
Q_{\omega, \vB} = e^{\ii \alpha} Q_{\omega, \vB}^\bullet(x+x_0) \quad \mbox{for almost every $x \in \R^n$},
$$
with some constants $\alpha \in \R$ and $x_0 \in \R^n$. As already mentioned for the proof of Theorem \ref{thm:sym} above, the key ingredient needed to be shown is that $\{ |\widehat{Q}_{\omega, \vB}| > 0 \}$ is a connected set. Luckily,  by exploiting the one-dimensionality of the problem, we can show that must have $\Omega \in \{ \R_{>0}, \R_{<0}, \R \}$, whence it follows that $\Omega$ is connected.

\begin{theorem}[Conjugation Symmetry for $n=1$] \label{thm:conj}
 Let $n = 1$ and suppose  the hypotheses of Theorem \ref{thm:existence} are satisfied.  Moreover, we assume $\sigma \in \N$ is an integer. Then any boosted ground state $Q_{\omega, \vB} \in H^{\frac 1 2}(\R)$ is of the form
 $$
Q_{\omega, \vB}(x) = e^{\ii \alpha} Q_{\omega, \vB}^\bullet(x+x_0) \quad \mbox{for a.\,e.~$x \in \R$},
$$
with some constants $\alpha \in \R$ and $x_0 \in \R$. In particular, any such $Q_{\omega, \vB} \in H^{\frac 1 2}(\R)$ satisfies (up to translation and phase) the conjugation symmetry (P2) for a.\,e.~$x \in \R$.
\end{theorem}

\begin{remarks*}
1) As in Theorem \ref{thm:sym} above, we actually obtain that $Q_{\omega, \vB}$ has non-negative Fourier transform. In particular, if $\widehat{Q}_{\omega, \vB} \in L^1(\R)$, we see that $Q_{\omega, \vB}$ (up to translation and phase) is a positive-definite function in the sense of Bochner.

2) For a conjugation symmetry result in general dimensions $n \geq 1$, we refer to our companion paper \cite{BuLeScSo-19}, where an analyticity condition on the Fourier symbol $p(\xi)$ is imposed in order to be able to deal with $n\geq 2$. 
\end{remarks*}

\subsection{Examples}

We list some essential examples, where we can deduce symmetries of boosted ground states for the following equation of the form \eqref{eq:gNLS}.

\begin{itemize}
\item \textbf{Fourth-order/biharmonic NLS} of the form
$$
\ii \pt_t u = \Delta^2 u + \mu \Delta u - |u|^{2 \sigma} u, \quad (t,x) \in \R \times \R^n,
$$
where $\mu \in \R$ and integer $\sigma \in \N$ with $1 \leq \sigma < \infty$ if $1 \leq n \leq 4$ and $1 \leq \sigma < \frac{4}{n-4}$ if $n \geq 5$.
\item \textbf{Fractional NLS} of the form
$$
\ii \pt_t u = (-\Delta)^s \, u - |u|^{2 \sigma} u, \quad (t,x) \in \R \times \R^n,
$$
with $s > 0$ and integers $\sigma \in \N$ such that $1 \leq \sigma < \sigma_*(s,n)$. 
\item \textbf{Half-Wave and Square-Root Klein-Gordon equations} of the form
$$
\ii \pt_t u = \sqrt{-\Delta+m^2} \, u - |u|^{2 \sigma} u, \quad (t,x) \in \R \times \R,
$$
with $m \geq 0$ and arbitrary integer $\sigma \in \N$.
 \end{itemize}

Finally, we also remark that the Fourier symmetrization techniques in this paper seem to be ready-made to be generalized to anisotropic NLS type equations, where the order of derivatives may depend on the spatial direction. For instance, we could study symmetries of boosted ground states for the focusing \textbf{half-wave-Schr\"odinger type equations} of the form
$$
\ii \pt_t u = \Delta_{x} u - \gamma \sqrt{-\Delta_y} u - |u|^{2 \sigma} u, \quad (t,x,y) \in \R \times \R_x^{k} \times \R_y^{l}
$$
with parameter $\gamma > 0$ and suitable integers $\sigma \in \N$. However, the relevant Sobolev space now becomes of the form 
$$
X = \{ u \in L^2(\R^{k+l}) : \int_{\R^{k+l}} (|\xi|^{2} + |\eta|) |\widehat{u}(\xi, \eta)|^2 \, d \xi \, d \eta < \infty \},
$$
where $\widehat{u}(\xi, \eta)$ with $(\xi, \eta) \in \R^k \times \R^l$ denotes the Fourier transform of $u$ in $\R^{n} = \R^{k} \times \R^l$.

\subsection*{Acknowledgments} The authors gratefully acknowledge financial support by the Swiss National Science Foundation (SNF) under grant no.~200021-149233. We also thank Tobias Weth for helpful comments on this work.

\section{Existence of Traveling Solitary Waves}

This section is devoted to the proof of Theorem \ref{thm:existence} by following the arguments in \cite{Hi-17}. Instead of concentration-compactness methods, we shall follow a different approach by adapting the techniques in \cite{BeFrVi-14} based on a general compactness lemma in $\dot{H}^s$ for general $s > 0$ (originally due to E.~Lieb for the case $s=1$).

\subsection{Proof of Theorem \ref{thm:existence}} We follow \cite{Hi-17} adapted to our setting here. Suppose that $P(D)$ satisfies Assumption \ref{ass1} with constants $s \geq \frac 1 2$, $A, B > 0$. Let $\vB \in \R^n$ with be given, where we additionally assume $|\vB| < A$ if $s=\frac 1 2$. Finally, we impose that $\omega > -\Sigma_\vB$ with $\Sigma_\vB$ defined in \eqref{def:Sigma_v}. Recalling that $P_\vB(D) = P(D) + \ii \vB \cdot \nabla$, we can define the norm 
$$
 \norm{u}_{\omega, \vB} := \left \langle u, (P_\vB(D)+\omega) u \right \rangle^{1/2} = \left ( \int_{\R^n} (p(\xi)- \vB \cdot \xi + \omega) |\widehat{u}(\xi)|^2 \, d\xi \right)^{1/2} .
$$
It is elementary to see that we have the norm equivalence
$$
\norm{u}_{\omega, \vB} \sim_{A,B,\vB,\omega} \norm{u}_{H^s}.
$$
Note that the functional $\JJ$ can be written as
$$
\JJ(u) = \frac{ \norm{u}_{\omega, \vB}^{2\sigma+2} }{\| u \|_{L^{2 \sigma +2}}^{2 \sigma +2} }.
$$
In what follows, we shall use $X \lesssim Y$ to mean that $X \leq C Y$ with some constant $C>0$ that only depends on $s,n,A,B,\sigma, \omega$. We set
$$
\JJJ := \inf\left  \{ \JJ(u) \mid  u \in H^s(\R^n), \; u \not \equiv 0 \right \}
$$
Since $0 < \sigma < \sigma_*(n,s)$, we obtain the Sobolev-type inequality
$$
\| u \|_{L^{2 \sigma+2}} \lesssim \norm{u}_{H^s} \lesssim \norm{u}_{\omega, \vB},
$$
which shows that $\JJJ > 0$ is strictly positive.

Suppose that $(u_j) \subset H^s(\R^n) \setminus \{ 0 \}$ is a minimizing sequence, i.\,e., we have $\JJ(u_j) \to \JJJ$ as $j \to \infty$. By scaling properties, we can assume without loss of generality that $\norm{u_j}_{L^{2 \sigma+2}} = 1$ for all $j \in \N$. Obviously, we find that $\sup_j \norm{u_j}_{\omega, \vB} \lesssim 1$. Hence the sequence $(u_j)$ is bounded in $H^s(\R^n)$. 

Next, we show that $(u_j)$ has  a non-zero weak limit in $H^s(\R^n)$, up to spatial translations and passing to a subsequence. To prove this claim, let us first assume that $s \neq n/2$ holds and therefore we have the continuous embedding $H^s(\R^n) \subset L^{2 \sigma_*+2}(\R^n)$. Now we choose a number $r \in (2 \sigma+2, 2 \sigma_* + 2)$.  By H\"older's and Sobolev's inequality, we have
\be \label{ineq:uj_one}
\norm{u_j}_{L^r} \leq \norm{u_j}_{L^{2 \sigma+2}}^{\theta} \norm{u_j}_{L^{2 \sigma_* + 2}}^{1-\theta} \lesssim \norm{u_j}_{L^{2 \sigma+2}}^{\theta} \norm{u_j}_{H^s}^{1-\theta},
\ee
with $\frac{\theta}{2 \sigma+2} + \frac{1-\theta}{2 \sigma_*+2} = \frac 1 r$. Since $\norm{u_j}_{L^{2 \sigma+2}} = 1$ for all $j$ and $\norm{u_j}_{L^2} \leq \norm{u_j}_{H^s} \lesssim 1$, we deduce from \eqref{ineq:uj_one} that there exist constants $\alpha, \beta, \gamma > 0$ such that
$$
\norm{u_j}_{L^2} \leq \alpha, \quad \norm{u_j}_{L^{2 \sigma+2}} \geq \beta, \quad \norm{u_j}_{L^r} \leq \gamma
$$
holds for all $j \in \N$. In the borderline case $s=n/2$, we also deduce the existence of such constants $\alpha, \beta, \gamma >0$, where we just have to replace $2 \sigma_*+2$ above by any number $q \in (2 \sigma+2, \infty)$ and use that $H^s(\R^n) \subset L^q(\R^n)$ holds. We omit the details.

Next, by invoking the Lemma \ref{lem:pqr}, we deduce that 
$$
\inf_{j \in \N} \left | \{ x \in \R^n \mid |u_j(x)| > \eta \} \right |  \geq c
$$
with some strictly positive constants $\eta, c> 0$, where $|\cdot|$ denotes the $n$-dimensional Lebesgue measure. Thus we can apply Lemma \ref{lem:lieb-compact} to conclude (after passing to a subsequence if necessary) that there exists a sequence of translations $(x_j)$ in $\R^n$ and some non-zero function $u \in H^s(\R^n) \setminus \{ 0 \}$ such that
\be
\mbox{$u_j(\cdot + x_j) \weakto u$ in $H^s(\R^n)$}.
\ee

Next, we show that the weak limit $u \not \equiv 0$ is indeed an optimizer for $\JJ$ and that $u_j \to u$ strongly in $H^s(\R^n)$. By the translational invariance of $\JJ$, we can assume that $x_j = 0$ for all $j$. Moreover, since the sequence $(u_j)$ is bounded in $H^s(\R^n)$, we can also assume pointwise convergence $u_j(x) \to u(x)$ almost everywhere. Recalling that $\norm{u_j}_{L^{2 \sigma+2}}=1$ for all $j$, the Br\'ezis-Lieb refinement of Fatou's lemma yields that
$$
\norm{u_j-u}_{L^{2 \sigma+2}}^{2 \sigma+2} + \norm{u}_{L^{2 \sigma+2}}^{2 \sigma+2} = 1 + o(1).
$$
Furthermore, from $\JJ(u_j)  \to \JJJ$ together with $\norm{u_j}_{L^{2 \sigma+2}}=1$ for all $j$ we conclude that
$$
\norm{u_j}_{\omega, \vB}^2 \to (\JJJ)^{\frac{1}{\sigma+1}}.
$$
On the other hand, since $u_j \weakto u$ in $H^s(\R^n)$ and writing $H = P_v(D) + \omega$ so that $\langle f, H f \rangle = \norm{f}_{\omega, \vB}^2$ for all $f \in H^s(\R^n)$, we readily find that 
$$
\langle u_j - u, H(u_j -u) \rangle + \langle u, H u \rangle = (\JJJ)^{\frac{1}{\sigma+1}} + o(1)
$$
by using elementary properties of the $L^2$-inner product. In summary, we thus deduce
\begin{align*}
& \JJJ \left \{ \norm{u_j-u}_{L^{2 \sigma+2}}^{2 \sigma+2} + \norm{u}_{L^{2 \sigma+2}}^{2 \sigma+2} + o(1) \right \}  = \JJJ \\
& = \left \{  \langle u_j - u, H(u_j -u) \rangle + \langle u, H u \rangle \right \}^{\sigma+1} \\
& \geq \langle u_j -u, H(u_j-u) \rangle^{\sigma+1} + \langle u, H u \rangle^{\sigma+1} + o(1) \\
& \geq \JJJ \norm{u_j-u}_{L^{2\sigma+2}}^{2 \sigma+2} + \langle u, H u \rangle^{\sigma+1} + o(1). 
\end{align*}
In the first inequality above, we used the elementary inequality $(x+y)^q \geq x^q + y^q$ for $x,y \geq 0$ and $q \geq 1$. Passing to the limit $j \to \infty$ and using that $u \not \equiv 0$, we obtain
$$
\JJJ \geq \frac{\langle u, H u \rangle^{\sigma+1}}{\norm{u}_{L^{2 \sigma+2}}^{2 \sigma+2}} = \JJ(u),
$$
which shows that $u \in H^s(\R^n) \setminus \{0 \}$ must be a minimizer. Also, we remark that we must have $\langle u_j - u, H(u_j-u) \rangle = \norm{u_j-u}_{\omega, \vB}^2 \to 0$ as $j \to \infty$, since equality must hold everywhere. This shows that in fact $u_j \to u$ strongly in $H^s(\R^n)$ due to the equivalence of norms $\norm{\cdot}_{H^s} \sim \norm{\cdot}_{\omega, \vB}$.

Finally, we note that an elementary calculation shows that any minimizer $Q_{\omega, \vB} \in H^s(\R^n) \setminus \{0 \}$ for $\JJ$ with $\norm{Q_{v, \omega}}_{L^{2 \sigma+2}}=1$ satisfies the corresponding Euler-Lagrange equation
\be
P_\vB(D) Q_{\omega, \vB} + \omega Q_{v,\omega} - (\JJJ)^{\frac{1}{\sigma+1}} |Q_{v,\omega}|^{2 \sigma} Q_{v, \omega} = 0.
\ee
After a rescaling $Q_{\omega, \vB} \mapsto \alpha Q_{\omega, \vB}$ with a suitable constant $\alpha > 0$, we find that $Q_{\omega, \vB}$ solves \eqref{eq:Qground}. This completes the proof of Theorem \ref{thm:existence}. \hfill $\qed$

\section{Rearrangements in Fourier Space}

In this section, we recall and introduce some notions needed to prove Theorems \ref{thm:sym} and \ref{thm:halfwave}.

\label{sec:rearrange}

\subsection{Preliminaries} We start by recalling some standard definitions in rearrangement techniques. Let $\mu_k$ denote the Lebesgue measure in dimension $k \geq 1$. For a Borel set $A \subset \R^k$, we denote by $A^*$ its symmetric rearrangement defined as the open ball $B_R(0)$ centered at the origin whose Lebesgue measure equals that of $A$, i.\,e., we set
$$
A^* = \{ x \in \R^k : |x |  < R \} \quad \mbox{such that $V_k R^k = \mu_k(A)$},
$$
where $V_k=\mu_k(B_1(0))$ is the volume of the unit ball in $\R^k$. Next, let $u: \R^k \to \C$ be measurable function that vanishes at infinity, which means that $\mu_k( \{ x \in \R^k : |u(x)| > t \} )$ is finite for all $t > 0$. We recall that the {\bf symmetric-decreasing rearrangement} of $u$ is defined as the nonnegative function $u : \R^k \to \R_+$ by setting
$$
u^*(x) = \int_0^\infty \chi_{\{ |u| > t \}^*}(x) \, \diff t,
$$
where $\chi_B$ denotes characteristic function of a the set $B \subset \R^k$. 

Let us now take $n \geq 2$ dimensions and decompose $\R^n= \R \times \R^{n-1}$. Accordingly, we write elements $x \in \R^n$ often as  $x = (x_1, x') \in \R \times \R^{n-1}$. For a measurable (Borel) function $u : \R^n \to \C$ vanishing at infinity, we define its {\bf Steiner symmetrization in $n-1$ codimensions}\footnote{We follow the nomenclature in \cite{Ca-14}.}. as the function $u^{*_1} : \R \times \R^{n-1} \to \R_+$ given by
$$
u^{*_1}(x_1, x') := u(x_1, \cdot)^*(x'),
$$ 
where $*$ on the right side denotes the symmetric-decreasing rearrangement of the function $x' \mapsto u(x_1, x')$ in $\R^{n-1}$ for each $x_1 \in \R$ fixed. Of course, the rearrangement operator $*^1$ can be easily generalized to arbitrary coordinate directions. More precisely, given a unit vector $\eB \in \mathbb{S}^{n-1}$, we pick a matrix $\Rs \in \mathrm{O}(n)$ such that $R \eB= \eB_1=(1, 0, \ldots, 0)$ and let $(\Rs u)(x) := f(\Rs^{-1} x)$ denote the action of $\Rs$ on functions $u: \R^n \to \C$. We can then define the Steiner symmetrization in $n-1$-dimensions with respect to $\eB$ as the nonnegative function $u^{*_e} : \R^n \to \R_+$ that is given by
$$
u^{*_\eB}:= \Rs^{-1}  ( (\Rs  u)^{*_1}) .
$$ 

Recalling the definition in \cite{LeSo-18}, we define the \textbf{Fourier rearrangement} of a function $u \in L^2(\R^n)$ to be given by
\be
u^\sharp := \mathcal{F}^{-1}  \left \{ ( \mathcal{F}(u) )^* \right \},
\ee
where $*$ denotes the symmetric-decreasing rearrangement in $\R^n$ and $\mathcal{F}$ is the Fourier transform 
\be
\mathcal{F}u(\xi) \equiv \widehat{u}(\xi) := \frac{1}{(2 \pi)^{n/2}} \int_{\R^n} u(x) e^{-i \xi \cdot x} \, \diff x,
\ee
defined for $u \in L^1(\R^n)$ and extended to $u \in L^2(\R^n)$ by density.
Finally, we come to the main technical tool used in this paper. Given a direction $\eB \in \mathbb{S}^{n-1}$ and $u \in L^2(\R^n)$, we define its \textbf{Fourier Steiner rearrangement in $n-1$ codimensions} by setting
\be
u^{\sharp_\eB} := \mathcal{F}^{-1} \left \{ \mathcal{F}(u)^{*_\eB} \right \}.
\ee
By a suitable rotation of coordinates in $\R^n$, it will often suffice to consider the case $\eB= \eB_1=(1,0,\ldots, 0)$ and likewise we simply write
\be
u^{\sharp_1} := \mathcal{F}^{-1} \left \{ ( \mathcal{F}(u)^{*_1} \right \}.
\ee

Next, we collect some basic properties of the operation $\sharp_\eB$ as follows.

\begin{lemma}\label{lem:basic_properties}
Let $n \geq 2$, $\eB \in \mathbb{S}^{n-1}$, and $u \in L^2(\R^n)$. Then the following properties hold.
\begin{enumerate}
\item[(i)] $\| u^{\sharp_\eB} \|_{L^2} = \| u \|_{L^2}$.
\item[(ii)] $u^{\sharp_\eB}$ is cylindrically symmetric with respect to $\eB$, i.\,e., for every matrix $\Rs \in \mathrm{O}(n)$ with $\Rs \eB= \eB$ it holds that
$$
u^{\sharp_\eB}(x) = u^{\sharp_\eB}(\Rs x) \quad \mbox{for a.\,e.~$x \in \R^n$}.
$$
\item[(iii)] If in addition $\widehat{u} \in L^1(\R^n)$, then $u^{\sharp_\eB}$ is a continuous and \textbf{positive definite} function in the sense of Bochner, i.\,e., we have
$$
\sum_{k,l=1}^m u^{\sharp_\eB}(x_k-x_l) \overline{z}_k z_l \geq 0
$$
for all integers $m \geq 1$ and $x_1, \ldots, x_m \in \R^n$ and $z \in \C^N$. In particular, it holds that
$$
u^{\sharp_\eB}(0) \geq |u^{\sharp_\eB}(x)| \quad \mbox{for all $x \in \R^n$}.
$$
\end{enumerate}
\end{lemma}

\begin{remark*}
Note that item (iv) says  in particular that $u^{\sharp_\eB}(0)$ is a real number. However, the values $u^{\sharp_\eB}(x)$ can be complex numbers for $x \neq 0$ in general.
\end{remark*}

\begin{proof}
Without loss of generality we can assume that $\eB = \eB_1=(1,0, \ldots,0)$. 

Item (i) follows from elementary arguments. Indeed, by Fubini's theorem, we find for any $f \in L^2(\R^n)$ that
\begin{align*}
\| f \|_{L^2}^2 & = \int_{\R^n} |f(x)|^2 \, dx = \int_{\R} \left ( \int_{\R^{n-1}} |f(x_1, x_2, \ldots, x_n)|^2 \, dx_2 \ldots d x_n \right ) \, dx_1 \\
& = \int_\R \left ( |f(x_1, x_2, \ldots, x_n)|^{*_1}|^2 \, dx_2 \ldots d x_n \right ) \, d x_1 = \| f^{*_1} \|_{L^2}^2,
\end{align*}
where used the equimeasurability of the functions $f(x_1, \ldots)$ and $f(x_1, \ldots)^{*_1}$ on $\R^{n-1}$ for every $x_1 \in \R$ fixed. By Plancherel's identity, we conclude that (i) is true.

Likewise, we see that (ii) holds true by elementary properties of the Fourier transform. Finally, we mention that (iii) follows from the fact that $\widehat{u^{\sharp_1}} = (\widehat{u}(\xi))^{*_1} \geq 0$ is non-negative and  classical arguments for positive-definite functions; see, e.\,g., \cite{ReSi-78}. 
\end{proof}

\subsection{Rearrangement Inequalities: Steiner meets Fourier}

Recall that the operator $P(D)$ is defined as $\widehat{(P(D) u)}(\xi) = p (\xi) \widehat{u}(\xi)$ through its real-valued multiplier $p : \R^n \to \R$. Furthermore, we recall that for the given velocity $\vB \in \R^n$ we define the operator
$$
P_\vB(D) = P(D) + \ii \vB \cdot \nabla,
$$ 
which has the Fourier symbol $p_\vB(\xi) = p(\xi) - \vB \cdot \xi$.

\begin{lemma} \label{lem:SteinerFourier}
Let $n \geq 2$. Suppose that $P(D)$ satisfies  Assumptions \ref{ass1} and \ref{ass2} with some $s \geq 1/2$. Let $\eB \in \mathbb{S}^{n-1}$ be some direction and assume that $\vB  \in \R^n$ is parallel to $\eB$. Then it holds that 
$$
\langle u^{\sharp_\eB}, P_\vB(D) u^{\sharp_\eB} \rangle \leq \langle u, P_\vB(D) u \rangle \quad \mbox{for all $u \in H^s(\R^n)$}.
$$
Moreover, we have equality if and only if $|\widehat{u}(\xi)| = (\widehat{u}(\xi))^{*_\eB}$ for almost every $\xi \in \R^n$.
\end{lemma}

\begin{proof}
By  a suitable rotation in $\R^n$, we can assume without loss of generality that $\eB=\eB_1=(1,0, \ldots,0)$ holds and thus $\vB= (|\vB|,0,\ldots,0)$. As before, we decompose $\xi \in \R^n$ as $\xi = (\xi_1, \xi') \in \R \times \R^{n-1}$. With some slight abuse of notation we can write $p(\xi) = p(\xi_1, |\xi'|)$ and $p_\vB(\xi)=p_\vB(\xi_1, |\xi'|) = p(\xi_1, |\xi'|) - |\vB| \xi_1$.

We adapt the following arguments in \cite{LeSo-18} to our setting here.

{\bf Step 1.} Suppose $A \subset \R^{n-1}$ is a measurable set with finite Lebesgue measure $\mu_{n-1}(A) < \infty$ in $n-1$ dimensions. For notational simplicity, we shall simply write $\mu$ instead of $\mu_{n-1}$ in the following. Let $A^*$ denote its symmetric-decreasing rearrangement in $\R^{n-1}$, i.\,e., the set $A^*=B_R(0) \subset \R^{n-1}$ is the open ball centered at the origin with measure $\mu(A^*)=\mu(A)$.  We claim that the following inequality holds
\be \label{ineq:Steiner1}
\int_{A^*} p_\vB(\xi_1, |\xi'|) \diff \xi' \leq \int_{A} p_\vB(\xi_1, |\xi'|) \diff \xi'
\ee
for any $\xi_1 \in \R$. Indeed, we have $\mu(A \setminus A^*) = \mu(A) - \mu(A \cap A^*)$ and $\mu(A^* \setminus A) = \mu(A^*) - \mu(A \cap A^*)$. Since $\mu(A)=\mu(A^*)$, we deduce that $\mu(A\setminus A^*) = \mu(A^* \setminus A)$. Next we recall that $|\xi'| \mapsto p(\xi_1, |\xi'|)$ is strictly increasing for all $\xi_1 \in \R$ fixed. Hence the map $|\xi'| \mapsto p_\vB(\xi_1, |\xi'|) = p(\xi'1, |\xi'|) - |\vB| \xi_1$ is strictly increasing as well. Since $|\xi'| \geq R$ for $\xi' \in A\setminus A^*$ and $|\xi'| < R$ for $\xi' \in A^* \setminus A$, this implies that
\begin{align}
\int_{A^* \setminus A} p_\vB(\xi_1, |\xi'|) \diff \xi' & \leq \int_{A^* \setminus A} p_\vB(\xi_1, R) \diff \xi ' = p_\vB(\xi_1, R) \mu(A^*\setminus A) \nonumber \\ &
= p_\vB(\xi_1, R) \mu(A \setminus A^*)  = \int_{A \setminus A^*} p_\vB(\xi_1, R) \diff \xi' \leq \int_{A \setminus A^*} p_\vB(\xi_1, \xi') \diff \xi' . \label{ineq:rearrnge_f}
\end{align}
Therefore we conclude
\begin{align*}
\int_{A^*} p_\vB(\xi_1, |\xi'|) \diff \xi' & = \int_{A^* \setminus A} p_\vB(\xi_1, |\xi'|) \diff \xi' + \int_{A^* \cap A} p_\vB(\xi_1, |\xi'|) \diff \xi \\
& \leq \int_{A \setminus A^*} p_\vB(\xi_1, |\xi'|) \diff \xi ' +  \int_{A^* \cap A} p_\vB(\xi_1, |\xi'|) \diff \xi = \int_{A} p_\vB(\xi_1, |\xi'|) \diff \xi ',
\end{align*}
which proves \eqref{ineq:Steiner1}.

{\bf Step 2.} Now let $f : \R^n \to \R_+$ be a nonnegative measurable function vanishing at infinity. We claim that
\be \label{ineq:Steiner2}
\int_{\R^n} f^{*_1} (\xi) p_\vB(\xi_1, |\xi'|) \diff \xi \leq \int_{\R^n} f(\xi) p_\vB(\xi_1, |\xi'|) \diff \xi ,
\ee 
where $f^{*_1}$ denotes the Steiner rearrangement in $n-1$ codimensions. To show the claimed inequality, we note that $f(\xi) = \int_0^\infty \chi_{\{ f > t\}}(\xi) \diff t$ by the layer cake representation and accordingly we have $f^{*_1}(\xi) = \int_0^\infty \chi_{\{ f > t \}^{*_1}}(\xi) \diff t$. Thus, by applying Fubini's theorem, we need to show that
\begin{align*}
\int_0^\infty \left ( \int_\R \left ( \int_{\R^{n-1}} \chi_{\{ f > t\}^{*_1}}(\xi_1, \xi') p_\vB(\xi_1, |\xi'|) \diff \xi' \right ) \diff \xi_1 \right ) \diff t & \leq \\
 \int_0^\infty \left ( \int_\R \left ( \int_{\R^{n-1}} \chi_{\{ f > t\}}(\xi_1, \xi') p_\vB(\xi_1, |\xi'|) \diff \xi' \right ) \diff \xi_1 \right ) \diff t .
\end{align*}
If we use \eqref{ineq:Steiner1} with the sets $B_{\xi_1} = \{ \xi' \in \R^{n-1} : f(\xi_1, \xi') > t \} \subset \R^{n-1}$ with $\xi_1 \in \R$, the definition of $*_1$ implies that $$\int_{\R^{n-1}} \chi_{\{ f > t\}^{*_1}}(\xi_1, \xi') p_\vB(\xi_1, |\xi'|) \diff \xi' \leq \int_{\R^{n-1}} \chi_{\{ f > t\}}(\xi_1, \xi') p_\vB(\xi_1, |\xi'|) \diff \xi'$$ for any $\xi_1 \in \R$. By integrating this inequality over $\xi_1$ and $t$, we arrive at the desired inequality stated in \eqref{ineq:Steiner2}.

{\bf Step 3.} By Plancherel's theorem and the definition of $u^{\sharp_1}$, the claimed inequality is equivalent to
$$
 \int_{\R^n} p_\vB(\xi) |(\widehat{u}(\xi))^{*_1}|^2 \diff \xi \leq \int_{\R^n} p_\vB(\xi) |\widehat{u}(\xi)|^2 \diff \xi.
$$
We now define the nonnegative function $f : \R^n \to \R_+$ with $f(\xi) = |\widehat{u}(\xi)|^2$. Clearly, $f$ is measurable and vanishes at infinity. Furthermore, we note that $f^{*_1}(\xi) = (|\widehat{u}(\xi)|^2)^{*_1}= |(\widehat{u}(\xi))^{*_1}|^2$, where the last equality follows from basic properties of the rearrangement $*_1$. By applying \eqref{ineq:Steiner2}, we obtain the claimed inequality stated in Lemma \ref{lem:SteinerFourier}.

{\bf Step 3.} Finally, we suppose that equality $\langle u^{\sharp_1}, P_\vB(D) u^{\sharp_1} \rangle = \langle u, P_\vB(D) u \rangle$ holds. Since $|\xi'| \mapsto p_\vB(\xi_1, |\xi'|)$ is strictly increasing, equality holds in \eqref{ineq:rearrnge_f} if and only if $\mu(A \setminus A^*) = 0$. Since $\mu(A) = \mu(A^*)$, this means that the sets $A$ and $A^*$ coincide (up to a set of measure zero). Therefore, by using the layer-cake representation for $f=|\widehat{u}|^2$ in \eqref{ineq:Steiner2}, we deduce the equality $f(\xi) = f^{*_1}(\xi)$ for almost every $\xi \in \R^n$, which is equivalent to $|\widehat{u}(\xi)| = (\widehat{u}(\xi))^{*_1}$ almost everywhere.

The proof of Lemma \ref{lem:SteinerFourier} is now complete.
\end{proof}

Next, we turn to a rearrangement inequality for $L^p$-norms. By arguing along the lines in \cite{LeSok}, we can prove the following result.

\begin{lemma} \label{lem:LpSteiner}
Let $n \geq 2$, $p \in 2 \N \cup \{ \infty \}$, and $\eB \in \mathbb{S}^{n-1}$. Then for all $u \in L^2(\R^n) \cap \FF(L^{p'}(\R^n))$ with $1/p+1/p'=1$, we have  $u^{\sharp_\eB} \in L^2(\R^n) \cap \FF(L^{p'}(\R^n))$ and 
$\| u \|_{L^p} \leq \| u^{\sharp_\eB} \|_{L^p}$. 
\end{lemma}

As a technical ingredient needed for the proof of Lemma \ref{lem:LpSteiner}, we need the following result concerning multiple convolutions in $\R^n$, which is a consequence of the classical {\em Brascamp-Lieb-Luttinger inequality}; see Lemma \ref{lem:BLL} below.

\begin{proposition}\label{prop:BLL}
Let $n \geq 2$, $\eB \in \mathbb{S}^{n-1}$, and $m \geq 2$. For any non-negative measurable functions $u_1, u_2,\dots, u_m:\R^n\to\R_+$ vanishing at infinity, we have
  \bes
  (u_1 \ast \ldots \ast u_m)(0)\leq (u_1^{*_\eB}\ast \ldots\ast u_m^{*_\eB})(0).
  \ees
\end{proposition}

\begin{proof}
Without loss of generality we can assume $\eB = \eB_1 =(1,0,\ldots,0) \in \R^n$. A calculation using Fubini's theorem yields
\begin{equation}
     \begin{split}
       & (u_1 \ast \cdots \ast u_m)(0) \\ 
     = & \int_{\R}\cdots\int_{\R} 
          I_{n-1} \bigg[u_1(y^1_1,\cdot),\dots,u_{m-1}(y^{m-1}_1,\cdot),u_m\bigg(-\sum_{i=1}^{m-1}y^i_1,\cdot\bigg)\bigg] 
          \diff y^1_1 \cdots \diff y^{m-1}_1.
   \end{split}
\end{equation}
  Here $I_{n-1}$ is defined according to \eqref{eq:quantity-in-BLL-corresponding-to-convolution-at-zero} with $B$ as the $(m-1)\times m$-matrix given by
  \bes 
    B = \left(\begin{array}{cccc|c}
1& 0 &\cdots &0 &-1\\
0&1&\cdots &0&-1\\
\vdots & &\ddots &\vdots & \vdots \\
0&0&\cdots &1&-1\\
\end{array}\right)
  \ees
and the matrix in the left block is the $(m-1)\times(m-1)$-unit matrix. By applying Lemma \ref{lem:BLL} with $d=n-1$ and recalling the definition of $*_1$, we deduce that
\bes 
   \begin{split}
       &   (u_1\ast \cdots\ast u_m)(0) \\
   =   &   \int_{\R}\cdots\int_{\R} I_{n-1} 
           \bigg[u_1(y^1_1,\cdot),\dots,u_{m-1}(y^{m-1}_1,\cdot),u_m\bigg(-\sum_{i=1}^{m-1}y^i_1,\cdot\bigg)\bigg]  
           \diff y^1_1 \cdots \diff y^{m-1}_1 \\
   \leq &  \int_{\R}\cdots\int_{\R} I_{n-1} 
           \bigg[u_1(y^1_1,\cdot)^*,\dots,u_{m-1}(y^{m-1}_1,\cdot)^*,u_m\bigg(-\sum_{i=1}^{m-1}y^i_1,\cdot\bigg)^*\bigg] 
           \diff y^1_1 \cdots \diff y^{m-1}_1 \\
    =   &  \int_{\R}\cdots\int_{\R} I_{n-1} 
           \bigg[u_1^{*_1}(y^1_1,\cdot),\dots,u_{m-1}^{*_1}(y^{m-1}_1,\cdot)^*,u_m^{*_1}
           \bigg(-\sum_{i=1}^{m-1}y^i_1,\cdot\bigg)\bigg] 
           \diff y^1_1 \cdots \diff y^{m-1}_1  \\
    =  &   (u_1^{*_1}\ast \cdots\ast u_m^{*_1})(0),
   \end{split}  
  \ees 
 where the last equality again follows from applying Fubini's theorem.
\end{proof}

\begin{proof}[Proof of Lemma \ref{lem:LpSteiner}] Without loss of generality we can assume that $\eB= \eB_1=(1,0,\ldots,0)$. The case of $p=2$ is clear. Let us assume $p=2m$ with some integer $m \geq 2$ so that the corresponding dual exponent is given by $p' = \frac{2m}{2m-1}$. Since $u \in L^p(\R^n)\cap \mathcal{F}(L^{p'}(\R^n))$, we can apply the version of the convolution lemma in \cite{LeSo-18} to conclude
\begin{equation}
\| u \|_{L^p}^p = \mathcal{F}(|u|^{2m})(0) = (\widehat{u}\ast \widehat{\widebar{u}} \ast \cdots \ast \widehat{u}\ast \widehat{\widebar{u}})(0),
\end{equation}
where the number of convolutions on the right-hand side equals $2m-1$. By Proposition \ref{prop:BLL}, we obtain that
\bes
(\widehat{u}\ast \widehat{\widebar{u}} \ast \cdots \ast \widehat{u}\ast \widehat{\widebar{u}})(0) \leq (\widehat{u}^{*_{1}}\ast (\widehat{\widebar{u}})^{*_{1}} \ast \cdots \ast (\widehat{u})^{*_{1}}\ast (\widehat{\widebar{u}})^{*_{1}})(0) = \mathcal{F}(|u^{\sharp_{1}}|^{2m})(0) = \| u^{\sharp_1} \|_{L^p}^p,
\ees
where we also used the fact that $\mathcal{F}(\widebar{u^{\sharp_{1}}}) = \mathcal{F}(\widebar{u})^{\ast_{1}}$ and the definition of $\sharp_1$. 

Finally, let us take $p=\infty$ and thus $p'=1$. We find, by using Fubini's theorem,
\begin{align*}
\norm{u}_{\infty} & \leq \int_{\R^d}|\widehat{u}(\xi)| \, d\xi = \int_{\R} \left ( \int_{\R^{n-1}} |\widehat{u}(\xi_1, \xi')| \, d \xi' \right ) d \xi_1 \\
&  =  \int_{\R} \left ( \int_{\R^{n-1}} \widehat{u}^{\sharp_1}(\xi_1, \xi') \, d\xi' \right ) d\xi_1 = u^{\sharp_{1}}(0).
\end{align*}
Since $\| u^{\sharp_1} \|_{L^\infty} = u^{\sharp_1}(0)$ holds by Lemma \ref{lem:basic_properties} (iii), we complete the proof.
\end{proof}

\section{Proof of Theorem \ref{thm:sym}}

We divide the proof of Theorem \ref{thm:sym} into two parts as follows. First, as the essential key point, we show that $\{ \xi \in \R^n : |\widehat{Q}_{\omega, \vB}(\xi)| > 0 \}$ is a connected set in $\R^n$. This fact then enables us to apply the recent rigidity result  \cite{LeSo-18} for the Hardy-Littlewood majorant problem in $\R^n$ to conclude the proof.

\subsection{Connectedness of the Set $\{ |\widehat{Q}_{\omega, \vB}| > 0\}$}
We start with with some notational preliminaries. Given two sets $X, Y \subset \R^n$, we shall use 
$$
X\oplus Y = \{ x + y : x \in X, \; y \in Y \}
$$ 
to denote their {\bf Minkowski sum}. Likewise, we denote their Minkowski difference by
$$
X \ominus Y = \{ x-y : x \in X, \; y \in Y \}.
$$ 
Furthermore, for a function $f : \R^n \to \R$ we use the short-hand notation
$$
\{ f > 0 \} = \{ x \in \R^n : f(x) > 0 \}
$$
throughout the following.

\begin{lemma} \label{lem:convo}
Let $f,g \in \R^n \to [0, \infty)$ be two non-negative and continuous functions. Assume that their convolution
$$
(f \ast g)(x) = \int_{\R^n} f(x-y) g(y) \diff y 
$$
has finite values for all $x \in \R^n$. Then it holds that 
$$
\{ f \ast g > 0 \} = \{ f > 0 \} \oplus \{ g > 0 \} .
$$
\end{lemma}

\begin{proof}
The proof is elementary. For the reader's convenience, we give the details. Let us write $\Omega_f = \{ f > 0 \}$, $\Omega_g = \{ g > 0 \}$ and $\Omega_{f \ast g} = \{ f \ast g > 0 \}$. We suppose that both $f \not \equiv 0$ and $g \not \equiv 0$, since otherwise the claimed result trivially follows.

First, we show that $\Omega_f \oplus \Omega_g \subset \Omega_{f \ast g}$. Let $x = x_1 + x_2$ with $x_1 \in \Omega_f$ and $x_2 \in \Omega_g$. By the continuity of $f$ and $g$, there exists some $\eps > 0$ such that $f > 0$ on $B_\eps(x_1)$ and $g > 0$ on $B_\eps(x_2)$. Thus, by using that $f \geq 0$ and $g \geq 0$ on all of $\R^n$, we get
$$
(f \ast g)(x)  = \int_{\R^n} f(x-y) g(y) \diff y \geq \int_{B_\eps(x_2)} f(x_1 + x_2 - y) g(y) \diff y  > 0,
$$
since $x_1 + x_2 - y \in B_\eps(x_1)$ when $y \in B_\eps(x_2)$. This shows that $\Omega_f \oplus \Omega_g \subset \Omega_{f \ast g}$.

Next, we prove that $\Omega_{f \ast g} \subset \Omega_f \oplus \Omega_g$ holds. Indeed, for every $x \in \R^n$, we can write
$$
(f \ast g)(x) = \int_{\R^n} f(x-y) g(y) \diff y = \int_{(\{x\} \ominus \Omega_f) \cap \Omega_g} f(x-y) g(y) \diff y,
$$ 
since $f(x-\cdot) \equiv 0$ on $\R^n \setminus (\{x \} \ominus \Omega_f)$ and $g \equiv 0$ on $\R^n \setminus \Omega_g$. However, if $x \not \in \Omega_f \oplus \Omega_g$ then $(\{x \} \ominus \Omega_f) \cap \Omega_g = \emptyset$. Thus $(f \ast g)(x) = 0$ for any $x \not \in \Omega_f \oplus \Omega_g$, whence it follows that the inclusion $\Omega_{f \ast g} \subset \Omega_f \oplus \Omega_g$ is valid.
\end{proof}

Next, we establish the following technical result in order to prove Theorem \ref{thm:sym}.

\begin{lemma} \label{lem:topo}
Let $n \geq 2$ and suppose $m \geq 2$ is an integer. Let $f \in L^{m/(m-1)}(\R^n) \geq 0$ be a continuous nonnegative function with $f=f^{*_\eB}$ with some $\eB \in \mathbb{S}^{n-1}$ and assume $f$ satisfies an equation of the form
\be \label{eq:supp}
 f(x) = h(x) \left (f \ast \ldots \ast f \right ) (x)  \quad \mbox{for all $x \in \R^n$}, 
\ee
with $m$ factors in the convolution product on the left side and $h : \R^n \to (0, +\infty)$ is some continuous positive function. Then the set $\{ f > 0 \} \subset \R^n$ is connected. 
\end{lemma}

\begin{proof}

Without loss of generality we can assume that $\eB = \eB_1 \in \mathbb{S}^{n-1}$ is the unit vector pointing in the $x_1$-direction. We denote the set
$$
\Omega = \{ x \in \R^n : f(x) > 0 \},
$$
where we assume that $\Omega \neq \emptyset$, since otherwise the result is trivially true. Let $\pi_1(\Omega)\subset \R \simeq \R \times \{0 \}$ be the projection of $\Omega \subset \R^n$ onto the $x_1$-axis, i.\,e., we set
$$
\pi_1(\Omega) = \{ x_1 \in \R : \mbox{$\exists x' \in \R^{n-1}$ with $(x_1, x') \in \Omega$} \}.
$$
Note that $\pi_1(\Omega)$ is an open subset of $\R$ because $\Omega \subset \R^n$ is open (by the continuity of $f$). 

Next, we recall that, for any $x_1 \in \R$ fixed, the sets $\{ x' \in \R^{n-1} : f(x_1, x')> 0 \}$ are open balls in $\R^{n-1}$ centered at the origin, due to the fact that $f=f^{\ast_1} \geq 0$, which implies that the map $x' \mapsto f(x_1, x')$ is radially symmetric in $\R^{n-1}$ and non-increasing in $|x'|$. Thus there exists a map
$$
\pi_1(\Omega) \to (0,+\infty], \quad x_1 \mapsto \rho(x_1)
$$
such that
$$
B_{\R^{n-1}}(0, \rho(x_1)) = \{ x' \in \R^{n-1} : f(x_1, x') > 0 \}
 $$
with the convention that $B_{\R^{n-1}}(0, +\infty) = \R^{n-1}$. In summary, we can write the set $\Omega$ in $\R^n$ as the union given by
\be \label{eq:Omega_union}
\Omega = \bigcup_{x_1 \in \pi_1(\Omega)} \{ x_1 \} \times B_{\R^{n-1}}(0,\rho(x_1))
\ee
with some strictly positive function $0 < \rho(x_1) \leq +\infty$ for $x_1 \in \pi_1(\Omega)$. 

Now, from the assumed equation satisfied by $f$, we deduce the set equality
\be \label{eq:Omega_mink}
\Omega = \bigoplus_{k=1}^m \Omega.
\ee
We claim that this implies that
\be \label{eq:pi_1}
\pi_1(\Omega) = \bigoplus_{k=1}^m \pi_1(\Omega) .
\ee
Indeed, the inclusion $\oplus_{k=1}^m \pi_1(\Omega) \subset \pi_1(\Omega)$ follows trivially from \eqref{eq:Omega_mink}. To see the reverse inclusion, let $x_1 \in \pi_1(\Omega)$ be given and thus $(x_1,0) \in \Omega$. By \eqref{eq:Omega_mink}, there exist points $(y_1, y_1'), \ldots, (y_m, y_m') \in \Omega$ such that $(x_1, 0) = (y_1, y_1') + \ldots + (y_m, y_m')$. However, from \eqref{eq:Omega_union} we deduce that $(y_1, 0), \ldots, (y_m,0) \in \Omega$ as well, whence it follows that $(x_1,0) = (y_1, 0) + \ldots + (y_m, 0)$. Therefore we have $\pi_1(\Omega) \subset \oplus_{k=1}^m \pi_1(\Omega)$ as claimed.

Finally, since $\pi_1(\Omega)$ is an open set in $\R$ satisfying \eqref{eq:pi_1},  we can invoke Lemma \ref{lem:Omega_1dim} below to deduce that we only can have the following three possibilities 
$$
\pi_1(\Omega) = \R, \quad \pi_1(\Omega) = (-\infty,0), \quad \mbox{or} \quad \pi_1(\Omega) = (0, +\infty).
$$
However, in either case, it is easy to see from \eqref{eq:Omega_union} that $\Omega$ must be connected, as any pair of points $(x_1, x') \in \Omega$ and $(y_1, y') \in \Omega$ can be connected by a continuous path in $\Omega$.

The completes the proof of Lemma \ref{lem:topo}.
\end{proof}

\subsection{Completing the Proof of Theorem \ref{thm:sym}}
Let $Q = Q_{\omega, \vB} \in H^s(\R^n)$ be a boosted ground state as in Theorem \ref{thm:sym}. 

It is elementary to check that $|Q|^{2 \sigma} Q \in L^1(\R^n)$ using that $\sigma \in (1, \sigma_*)$. Hence by \eqref{eq:Qground} and taking the Fourier transform, we conclude that $\widehat{Q}(\xi) = \frac{1}{p_\vB(\xi) + \omega} \widehat{(|Q|^{2 \sigma} Q)}(\xi)
$ is a continuous function due to the assumed continuity of $p(\xi)$. Next, by Lemma \ref{lem:SteinerFourier} and \ref{lem:LpSteiner}, we conclude that $Q^{\sharp_1}$ is also a boosted ground state and it must hold that
$$
|\widehat{Q}(\xi)| = (\widehat{Q}(\xi))^{*_{\eB}}  \quad \mbox{for all $\xi \in \R^n$}.
$$
By  writing the equation \eqref{eq:Q} in Fourier space, we find that  the set
$$
\Omega = \{ \widehat{Q}^{*_1} > 0 \} = \{ |\widehat{Q}(\xi)| > 0 \}
$$ 
is  a connected set in $\R^n$ by using Lemma \ref{lem:topo} with $f = |\widehat{Q}|^{*_1}$ and $h = (p_\vB(\xi)+ \omega)^{-1}$.

Finally, since $Q$ and $Q^{*_\eB}$ are both boosted ground states, we must also have the equality $\| Q \|_{L^p} = \| Q^{*_\eB} \|_{L^p}$. We can now invoke Lemma \ref{lem:HLM} to deduce that
$$
\widehat{Q}(\xi) = e^{\ii (\alpha + \beta \cdot \xi)} \widehat{Q}^{*_\eB}(\xi) \quad \mbox{for all $\xi \in \R^n$},
$$
with some constants $\alpha \in \R$ and $\beta \in \R^n$. Hence it follows that $Q(x) = e^{\ii \alpha} Q^{\sharp_1}(x+x_0)$ for almost every $x \in \R^n$, where $\alpha \in \R$ and $x_0 \in \R^n$ are some constants.

The proof of Theorem \ref{thm:sym} is now complete. \hfill $\qed$

\section{Proof of Theorem \ref{thm:conj}}

Let the hypotheses of Theorem \ref{thm:conj} be satisfied and suppose $Q=Q_{\omega, \vB} \in H^{s}(\R)$ is a boosted ground state. As before, we consider the set
$$
\Omega = \{ \xi \in \R: |\widehat{Q}(\xi)| > 0 \} .
$$
Similarly as in the proof of Theorem \ref{thm:sym}, we conclude that $\widehat{Q}$ is a continuous function (and hence $\Omega$ is open). Moreover, it is elementary to see that (using that $\sigma \in \N$)
$$
\langle Q^\bullet, P(D) Q^\bullet \rangle \leq \langle Q, P(D) Q \rangle \quad \mbox{and} \quad \| Q \|_{L^{2 \sigma+2}} \leq \| Q^\bullet \|_{L^{2 \sigma+2}} ,
$$
see \cite{BuLeScSo-19}[Lemma 2.1]. Hence  we conclude that $Q^\bullet \in H^s(\R)$ is also a boosted ground state with $\| Q \|_{L^2} = \| Q^\bullet \|_{L^2}$. Furthermore, by arguing in the same way as in the proof of Theorem \ref{thm:sym}, we deduce that
\be \label{eq:omega_mink}
\Omega = \bigoplus_{k=1}^{2 \sigma +1} \Omega,
\ee
which means that the set $\Omega \subset \R$ is identical to its $(2\sigma+1)$-fold Minkowski sum. Using the one-dimensionality of the problem, we can now prove the following auxiliary result.

\begin{lemma}\label{lem:Omega_1dim}
Suppose $\Omega \subset \R$ is an open and non-empty set such that 
$$
\Omega = \bigoplus_{k=1}^m \Omega
$$
for some integer $m \geq 2$. Then it holds that 
$$
\Omega \in \{ \R_{>0}, \R_{<0}, \R \}.
$$
\end{lemma}

\begin{remark*}
For higher dimensions $\Omega \subset \R^n$ when $n \geq 2$, we conjecture that $\Omega$ is always a connected set.
\end{remark*}

\begin{proof}
We split the proof into the following steps.

\medskip
\textbf{Step 1.} Let us first suppose that $\Omega \subset \R_{\geq 0}$ holds. We claim that we necessarily  have
\be
\Omega =  \R_{>0}.
\ee 
To see this, we first show that
\be \label{eq:OmInf}
\inf \Omega = 0.
\ee
Indeed, let us denote $x_* = \inf \Omega \geq 0$. For every $\eps > 0$, we can find $x \in \Omega$ such that $x_* \leq x < x_* + \eps$. Since $\Omega = \oplus_{k=1}^m \Omega$, we can find $x_1, \ldots, x_m \in \Omega$ such that $x = \sum_{k=1}^m x_k$ and, of course, we have $x_k \geq x_*$ for $k=1, \ldots, m$. Thus we conclude 
$$
m x_* \leq \sum_{k=1}^m x_k = x < x_* + \eps.
$$
Therefore we find that  
$$
(m-1) x_* < \eps.
$$
Since $\eps > 0$ is arbitrary, we deduce that \eqref{eq:OmInf} holds.

Next, we show that $\Omega$ is an open connected set in $\R$ (and hence it is an open interval since we are in one dimension). We argue by contradiction. Suppose $\Omega$ is not connected, i.\,e., we can find $x, y \in \Omega$ with $x < y$ and some $b \in (x,y)$ such that $b \not \in \Omega$. Moreover, since $\Omega$ is open, we can always arrange that $b$ is chosen such that
\be
(x,b) \subset \Omega \quad \mbox{and} \quad b \not \in \Omega.
\ee
Recalling that $\inf \Omega = 0$ we can now find some $c \in \Omega$ with $0 < c < \frac{b-x}{m-1}$. Hence it follows
\be
x + (m-1) c < b \quad \mbox{and} \quad b + (m-1) c > b.
\ee
Thus there exists $d \in (x,b) \subset \Omega$ with $d + (m-1) c = b$. Since $\Omega = \oplus_{k=1}^m \Omega$, we deduce from this that we have $b \in \Omega$ too. But this is a contradiction. Hence the open set $\Omega \subset \R$ is connected, i.\,e., we have
$$
\Omega = (\inf \Omega, \sup \Omega) = (0, \sup \Omega)
$$
since $\inf \Omega = 0$. From the assumed Minkowski-sum property of $\Omega$ it is easy to see that $\sup \Omega = +\infty$. Thus we conclude $\Omega= (0, +\infty) = \R_{>0}$, provided that $\Omega \subset \R_{\geq 0}$ holds. Likewise, we can show that $\Omega = \R_{< 0}$ whenever $\Omega \subset \R_{\leq 0}$.

\medskip
\textbf{Step 2.} It remains to discuss the case when both $\Omega \cap \R_{\geq 0} \neq \emptyset$ and $\Omega \cap \R_{\leq 0} \neq \emptyset$. In this case, we first claim that there exist numbers $\underline{y} < 0$ and $\overline{y} > 0$ such that
\be \label{eq:intervals}
(-\infty, \underline{y}) \cup (\overline{y}, +\infty) \subset \Omega.
\ee
Indeed, by assumption on $\Omega$, exist real numbers $y_- < 0$ and $y_+ > 0$ such that $y_-, y_+ \in \Omega$. Since $\Omega$ is open, we find $B_\eps(y_-) \subset \Omega$ and $B_\eps(y_+) \subset \Omega$ for some $\eps >0$. Let us introduce the integer $m=2\sigma+1 \geq 2$. From the elementary fact $B_{r_1}(x_1) \oplus B_{r_2}(x_2) = B_{r_1+r_2}(x_1+x_2)$ for the Minkowski sum of two open balls together with \eqref{eq:omega_mink}, we deduce
$$
\bigoplus_{k=1}^m B_\eps(y_+) = B_{m\eps}(my_+) \subset \Omega.
$$
Using this fact inductively and \eqref{eq:omega_mink}, we obtain a sequence of intervals $\{ I_{n} \}_{n=1}^\infty$ with $I_n \subset \Omega$ that are given by the recursion formula
$$
\left \{ \begin{array}{ll} I_{n+1} = B_{(m-1)\eps}((m-1) y_+) \oplus I_n & \quad \mbox{for $n \geq 1$},\\
I_1 = B_\eps(y_+). \end{array} \right .
$$
Hence we have
\begin{align*}
I_{n+1} & = B_{(m-1)\eps}((m-1) y_+) \oplus B_\eps(y_+) \oplus_{k=1}^{m-1} B_{(m-1) \eps}((m-1) y_+) \\
& = B_{(n(m-1) + 1)\eps}((n(m-1) + 1) y_+).
\end{align*}
Now we claim that 
\be \label{eq:overlap}
I_{n+1} \cap I_{n+2} \neq \emptyset  \quad \mbox{for $n \geq n_0$},
\ee
where $n_0 \geq 1$ is sufficiently large. This is true if 
$$
((n+1)(m-1)+1)y_+ - ((n+1)(m-1)+1) \eps \leq (n(m-1)+1)y_+ + (n(m-1) +1) \eps,
$$
which in turn is equivalent to
$$
\left ( 2n + 1+ \frac{2}{m-1} \right ) \eps \geq y_+.
$$
Evidently, this holds if $n \geq n_0$ with some sufficiently large integer $n_0 \in \N$.

By \eqref{eq:overlap}, we deduce that $I = \cup_{n \geq N} I_{n+1} \subset \Omega$ is an (open) interval and it is elementary to check that $\sup I = +\infty$. Hence we conclude that $I = (\overline{y}, +\infty) \subset \Omega$ for some $\overline{y} > 0$. Likewise, we show that $(-\infty, \underline{y}) \subset \Omega$ for some $\underline{y} < 0$. This proves \eqref{eq:intervals}.

Finally, we pick a positive number $c > \max \{\overline{y}, -\underline{y} \}$ and note that $c \in (\overline{y}, +\infty) \subset \Omega$. Next, we define the negative number $\tilde{c} = -(m-1) c <0$. Since $m \geq 2$, we have that $\tilde{c} \in (-\infty, \underline{y}) \subset \Omega$. But by the Minowski-sum property of $\Omega$, we conclude that 
$$
0 = \tilde{c} +(m-1) c =  \tilde{c} + \sum_{k=1}^{m-1} c \in  \bigoplus_{k=1}^m \Omega = \Omega.
$$ 
Hence $0 \in \Omega$ and we deduce that $B_r(0) \subset \Omega$ for some $r > 0$, since $\Omega$ is open. By \eqref{eq:omega_mink} and the Minkowski sums of balls, this implies that $B_{mr}(0) \subset \Omega$. Thus by iteration we obtain 
$$
B_{Nmr}(0) \subset \Omega \quad \mbox{for any $N \in \N$}.
$$
By taking $N \in \N$ arbitrarily large, we conclude that $\Omega = \R$ holds.

The proof of Lemma \ref{lem:Omega_1dim} is now complete.
\end{proof}

With Lemma \ref{lem:Omega_1dim} at hand, we can now finish the proof of Theorem \ref{thm:conj} as follows. Since we must have equality $\| Q_{\omega, \vB} \|_{L^{2 \sigma+2}} = \| Q_{\omega, \vB}^\bullet \|_{L^{2 \sigma+2}}$ for any boosted ground state $Q_{\omega, \vB} \in H^{s}(\R)$, we deduce from Lemma \ref{lem:HLM} below that the conclusion of Theorem \ref{thm:conj} holds. \hfill $\qed$

\appendix
\section{Some Technical Results}

  \begin{lemma} [pqr Lemma; see {\cite{FrLiLo-86}}] \label{lem:pqr}
Let $(\Omega,\Sigma,\mu)$ be a measure space. Let $1\leq p<q<r\leq\infty$ and let $C_p, C_q, C_r > 0$ be positive constants. Then there exist constants $\eta, c>0$ such that, for any measurable function $f\in L^p_\mu(\Omega)\cap L^r_\mu(\Omega)$ satisfying
  \bes 
    \LpNmu{p}{f}^p\leq C_p, \quad \LpNmu{q}{f}^q\geq C_q,\quad \LpNmu{r}{f}^r\leq C_r,
  \ees
it holds that 
  \bes 
    d_f(\eta) := \mu(\{x\in \Omega;\  |f(x)|>\eta\}) \geq c.
  \ees
The constant $\eta>0$ only depends on $p, q, C_p, C_q$ and the constant $c>0$ only depends on $p,q,r, C_p, C_q, C_r.$  
\end{lemma}
\begin{proof} 
See \cite[Lemma 2.1]{FrLiLo-86}.
\end{proof}

\begin{lemma}[Compactness modulo translations in $\dot{H}^s(\R^n)$; see {\cite{BeFrVi-14}}] \label{lem:lieb-compact}
Let $s>0,$ $1<p<\infty$ and $(u_j)_{j\in\N}\subset \dot{H}^s(\R^n)\cap L^p(\R^n)$ be a sequence with
  \bes
    \sup_{j\in\N} \left( \dHsN{s}{u_j} + \LpN{p}{u_j} \right) < \infty,
  \ees
and, for some $\eta, c>0$ (with $|\cdot|$ being Lebesgue measure)
  \bes
     \inf_{j\in\N} |\{ x\in\R^n;\ |u_j(x)|>\eta \}| \geq c. 
  \ees
Then there exists a sequence of vectors $(x_j)_{j\in\N}\subset \R^n$ such that the
translated sequence $u_j(x+x_j)$ has a subsequence that converges weakly in $\dot{H}^s(\R^n)\cap L^p(\R^n)$ to a nonzero function $u\nequiv 0.$
\end{lemma}
\begin{proof} See \cite[Lemma 2.1]{BeFrVi-14}. 
\end{proof}

\begin{lemma}[Brascamp--Lieb--Luttinger Inequality] \label{lem:BLL}
Let $d \geq 1$ and $m \geq 2$ be integers. Suppose that  $u_1, u_2, \dots, u_m: \R^d\to \R_+$ are
nonnegative measurable functions vanishing at infinity. Let $1 \leq k\leq m$ and $B=[b_{ij}]$ be a given $k\times m$ matrix 
 (with $1\leq i\leq k,$ $1\leq j\leq m).$ If we define
  \be \label{eq:quantity-in-BLL-corresponding-to-convolution-at-zero}
  I_d[u_1,\dots, u_m] := \int_{\R^d} \cdots \int_{\R^d} 
  \prod_{j=1}^{m} u_j\bigg(\sum_{i=1}^{k} b_{ij} y^i \bigg)\diff y^1 \cdots \diff y^{k},
  \ee
then it holds that
  \bes
  I_d[u_1,\dots, u_m] \leq I_d[u_1^*,\dots, u_m^*],
  \ees
  where $*$ denotes the symmetric-decreasing rearrangement in $\R^d$.
 \end{lemma}

We recall from \cite{LeSo-18} the following result.

\begin{lemma}[Equality in the Hardy-Littlewood Majorant Problem in $\R^n$] \label{lem:HLM}
 Let $n \geq 1$ and $p \in 2 \N \cup \{\infty\}$ with $p> 2$. Suppose that $f, g \in \mathcal{F}(L^{p'}(\R^n))$ with $1/p+1/p'=1$ satisfy the majorant condition
$$
|\widehat{f}(\xi)| \leq \widehat{g}(\xi) \quad \mbox{for a.\,e.~$\xi \in \R^n$.}
$$
In addition, we assume that $\widehat{f}$ is continuous and that $\{ \xi\in \R^n : |\widehat{f}(\xi)| > 0 \}$ is a connected set. Then equality
$$
\| f \|_{L^p} = \| g \|_{L^p}
$$
holds if and only if
$$
\widehat{f}(\xi) = e^{\ii (\alpha + \beta \cdot \xi)} \widehat{g}(\xi) \quad  \mbox{for all $\xi \in \R^n$},
$$
with some constants $\alpha \in \R$ and $\beta \in \R^n$.
\end{lemma}

\bibliographystyle{amsplain}


\begin{thebibliography}{}


\bibitem{BeFrVi-14}
Jacopo Bellazzini, Rupert L. Frank and Nicola Visciglia,
\emph{Maximizers for Gagliardo-Nirenberg inequalities and related non-local problems},
Math. Ann. \textbf{360} (2014), no. ~3-4, 653--673.

\bibitem{BeGeLeVi-19}
Jacopo Bellazzini, Vladimir Georgiev, Enno Lenzmann, and Nicola Visciglia,
\emph{On Traveling Solitary Waves and Absence of Small Data Scattering for Nonlinear Half-Wave Equations},
Comm. Math. Phys. (2019), in press, \url{https://doi.org/10.1007/s00220-019-03374-y}.

\bibitem{BoCa-18}
    Denis Bonheure, Jean-Baptiste Casteras, Edeson Moreira dos Santos, and Robson Nascimento,
     \emph{Orbitally stable standing waves of a mixed dispersion
              nonlinear {S}chr\"{o}dinger equation},
   SIAM J. Math. Anal.~\textbf{50} (2018), no.~5, 5027--5071.
   
\bibitem{BuLeScSo-19}
Lars Bugiera, Enno Lenzmann, Armin Schikorra, and J\'er\'emy Sok,
\emph{On Symmetry and Uniqueness of ground state for linear and nonlinear elliptic PDEs},
Preprint (2019).

 
\bibitem{BoHiLe-16}
  Thomas Boulenger,  Dominik Himmelsbach,  and Enno Lenzmann,
   \emph{Blowup for fractional {NLS}},
   J. Funct. Anal. \textbf{271} (2016), no.~9, 2569--2603.
  
\bibitem{Brascamp-Lieb-Luttinger:A-General-Rearrangement-Inequality}
Herm J. Brascamp, Elliott H. Lieb and Joaquin M. Luttinger,
\emph{A general rearrangement inequality for multiple integrals},
J. Functional Analysis \textbf{17} (1974), 227--237.
\MR{0346109}

\bibitem{Brezis-Lieb:A-relation-between-pointwise-conv-of-fcts-and-conv-of-fctals}
Ha\"im Br{\'e}zis and Elliot Lieb,
\emph{A relation between pointwise convergence of functions and convergence of functionals},
Proc. Amer. Math. Soc. \textbf{88} (1983), no. ~3, 486--490.


\bibitem{Ca-14}
Guiseppe M. Capriani,
\emph{The Steiner rearrangement in any codimension},
Calc. Var. Partial Differential Equations \textbf{49} (2014), no. ~1-2, 517--548.
\MR{3148126}


\bibitem{Fi-15}
    Gadi Fibich,
     \emph{The nonlinear {S}chr\"{o}dinger equation}, Applied Mathematical Sciences, 192, Springer, Cham, 2015.
     
    
 \bibitem{FiIlPa-02}
    Gadi Fibich, Boaz Ilan, and George Papanicolaou,
    \emph{Self-focusing with fourth-order dispersion},
   SIAM J. Appl. Math. \textbf{62} (2002), no.~4, 1437--1462.

\bibitem{FrLiLo-86}
   J\"urg Fr\"{o}hlich, Elliott H. Lieb, and Michael Loss,
     \emph{Stability of {C}oulomb systems with magnetic fields. {I}.
              {T}he one-electron atom},
   Comm. Math. Phys.~\textbf{104} (1986), no.~2, 251--270.
 
	
\bibitem{GeLePoRa-18}
Patrick G\'erard, Enno Lenzmann, Oana Pocovnicu, and Pierre Rapha\"el,
\emph{A two-soliton with transient turbulent regime for the cubic half-wave equation on the real line},
Annals of PDE \textbf{4} (2018), Art.~7, 166pp.

\bibitem{Hi-17}
Dominik Himmelsbach, 
\emph{Blowup, solitary waves and scattering for the fractional nonlinear Schr\"odinger equation}, PhD Thesis (2017), University of Basel, \url{http://edoc.unibas.ch/diss/DissB_12432}.

\bibitem{HoSi-15}
Younghun Hong and Yannick Sire,
\emph{On Fractional Schr{\"o}dinger Equations in Sobolev Spaces}, 
Commun. Pure Appl. Anal. \textbf{14} (2015), no.~6, 2265--2282. 

\bibitem{KaSh-00}
    V. I. Karpman and A. G. Shagalov,
    \emph{Stability of solitons described by nonlinear
              {S}chr\"{o}dinger-type equations with higher-order dispersion},
   Phys. D \textbf{144} (2000), no.~1-2, 194--210.
  
\bibitem{KrLeRa-13}
Joachim Krieger, Enno Lenzmann and Pierre Rapha{\"e}l,
\emph{Nondispersive solutions to the ${L}^2$-critical Half-Wave Equation},
Arch. Rational Mech. Anal. \textbf{209} (2013), 61--129.

 \bibitem{LeSo-18}
   Enno Lenzmann and J\'er\'emy Sok, \emph{A sharp rearrangement principle in Fourier space and symmetry results for PDEs with arbitrary order},
   Preprint (2018), arXiv:1805.06294.
   
   \bibitem{Lieb-Loss:Analysis}
Elliott H. Lieb and Michael Loss, 
\emph{Analysis} (Second edition),
American Mathematical Society, Providence, RI, 2001.
\MR{1817225}

\bibitem{NaRa-18}
Ivan Naumkin and Pierre Rapha\"el,
\emph{Om small traveling waves to the mass critical fractional NLS},
Calc. Var. Partial Differential Equations \textbf{57} (2018), no.~3, 36pp.
 

\bibitem{ReSi-78}
    M. Reed and B. Simon, \emph{Methods of modern mathematical physics. {II}. Fourier analysis, self-adjointness}, Academic Press, New York-London, 1975.



\end{thebibliography}

\end{document}